\documentclass{amsart}
\usepackage{amsmath}
\usepackage{amssymb}
\usepackage{graphicx}
\input xy
\xyoption{all}
\newtheorem{theorem}{Theorem}[section]

\newtheorem{lemma}{Lemma}[section]
\newtheorem{corollary}{Corollary}[section]

\theoremstyle{definition}

\begin{document}
    \title[On multiplicative subgroups in division rings]{On multiplicative subgroups in division rings}
    \author[Bui Xuan Hai]{Bui Xuan Hai}\author[Nguyen Anh Tu]{Nguyen Anh Tu}
    \address{University of Science, VNU-HCMC, Vietnam}
   \email{bxhai@hcmus.edu.vn;  anhtu812@yahoo.com}
\address{Faculty of Mathematics and Computer Science, University of Science, VNU-HCMC, 227 Nguyen Van Cu Str., Dist. 5, Ho Chi Minh City, Vietnam.}

\keywords{ maximal subgroups, subnormal subgroups,  $FC$-elements. \\
\protect \indent 2010 {\it Mathematics Subject Classification.} 16K20.}

 \maketitle

 \begin{abstract} Let $D$ be a division ring. In this paper, we investigate properties of  subgroups  of an arbitrary  subnormal subgroup of the multiplicative group $D^*$ of $D$. The new obtained results  generalize some previous results on subgroups of $D^*$.   
\end{abstract}
\section{Introduction}        
Let $D$ be a division ring with center $F$ and  the multiplicative group  $D^*$. The  subgroup structure  of $D^*$ is one of subjects which attract the attention  of many authors around the world (see, for example, \cite{maxn}, \cite{exist}, \cite{nilpotent}-\cite{hai-thin2}, \cite{her3}, \cite{her4}, \cite{huz}, \cite{exist1}, \cite{fg}, \cite{free},  \cite{scott1}, \cite{stuth} etc.).   Some well-known classical results show that subnormal subgroups of $D^*$ behave as $D^*$ in several ways. Thus, from the earlier result of Scott  \cite{scott1}, we see that there does not exist any abelian non-central normal subgroup in a non-commutative division ring. This is a particular case of the most important result concerning subnormal subgroup structure  obtained later  by Stuth \cite{stuth} asserting  that every soluble subnormal subgroup of $D^*$ is central. In \cite{her1}, Herstein  proved that $x^{D^*}$ is infinite for  every non-central element $x$ of $D^*$. This result was extended by Scott in \cite{scott1}, where he proved that $|x^{D^*}|=|D|$. Moreover, in this work, Scott showed that if $G$ is a non-central subnormal subgroup, then for every non-central element $x$ of $D$, the division subring generated by $x^G$ is $D$. There are also some other results showing that subnormal subgroups of $D^*$  are, roughly speaking,  `` big". For further information, we  refer to \cite{maxn}, \cite{hai-thin1}-\cite{hai-thin2},  \cite{her3},  \cite{her4}, \cite{huz}, \cite{fg} and  references therein.

In this paper, we study  subgroups of an arbitrary  subnormal subgroup of $D^*$, especially its maximal subgroups. We refer to  \cite{exist}, \cite{exist2}, \cite{exist1}, and references therein for information on the existence of maximal subgroups in non-commutative division rings. Recall that in \cite{maxn}, \cite{nilpotent}, \cite{free}, Akbari et al., and Mahdavi-Hezavehi study maximal subgroups of $D^*$  and many nice properties of such subgroups were obtained. In the present paper, studying  maximal subgroups  of $G$, we get in various cases  the similar results for these subgroups as the results obtained previously in \cite{maxn}, \cite{nilpotent}, \cite{free} for maximal subgroups of $D^*$. 

Throughout this paper, for a ring $R$ with identity $1\neq 0$, the symbol  $R^* $ denotes the group of all units in $R$. If $S$ is  a non-empty subset  of a division ring $D$,  then $F[S]$ and  $F(S)$ denote respectively the subring and the division subring of $D$ generated by the set $F\cup S$.  For a given group $G$ and its subgroup $H$, we denote the derived group of $G$ and the core of $H$ in $G$  respectively by $G'$ and $H_G:=\bigcap_{x\in G} xHx^{-1}$. Also, if $x\in G,$ then $x^H:=\{x^h=hxh^{-1}, h\in H\}$. If $x^G$ is finite, then we say that $x$ is  an {\em $FC$-element} of $G$. The set of all $FC$-elements of $G$ is called the {\em $FC$-center} of $G$. If $x, y\in G$, then  $[x,y]:=xyx^{-1}y^{-1}$, and $[H,K]$ is the subgroup of $G$ generated by all elements $[h, k],  h\in H,\,k\in K$. An element $x\in D$ is said to be {\em radical} over $F$ if there exists a positive integer $n(x)$ depending on $x$ such that $x^{n(x)}\in F$. A non-empty subset $S$ of $D$ is {\em radical} over $F$ if every element of $S$ is radical over $F$.  If $A$ is a ring or a group, then the symbol $Z(A)$ denotes the center of $A$. All other notation and symbols in this paper are standard and one can find, for example, in \cite{lam}, \cite{rob}, \cite{scott2}, \cite{ilg}, \cite{ls}.

\section{Algebraicity over a division subring}

Let $D$ be a division ring with center $F$. Assume that $A$ is a conjugacy class of $D$ which is algebraic over $F$ with the minimal polynomial $f(t)\in F[t]$ of degree $n$. Then, there exist $a_1, \ldots, a_n\in A$ such that
$$f(t)=(t-a_1)\ldots(t-a_n)\in D[t].$$

This factorization theorem which is due to Wedderburn,  plays an important role in the theory of polynomials over a division ring and its applications in the study of the structure of division rings are well-known.  In this section, we give some analogue of this theorem which will be used for our study in next sections of the paper. 

Let $K\subseteq D$ be a pair of  division rings and $\alpha \in D$. We say that $\alpha$ is {\em right (resp. left) algebraic} over $K$ if there exists some non-zero polynomial $f(t)\in K[t]$ having $\alpha$ as a right (resp. left) root. A monic polynomial from $K[t]$ with smallest degree having $\alpha$ as a right (resp. left) root is called a {\em right (resp. left) minimal polynomial} of $\alpha$ over $K$. Since throughout this paper we consider only right roots and the right algebraicity, we shall always omit the prefix ``right". The minimal polynomial of $\alpha$ over $K$ is unique, but, it may not be irreducible as the following example shows: 

Let $\mathbb{H}$ be the division ring of real quaternions. Then, $f(t)=t^2+1\in \mathbb{C}[t]$ is the minimal polynomial of $j$ and $k$ over $\mathbb{C}$.  Here, $\{1, i, j, k\}$ is the standard basis of $\mathbb{H}$ over $\mathbb{R}$.

We note that the proof of the following lemma is a simple modification of  the proof of Lemma (16.5) in \cite{lam}. 

\begin{lemma}\label{lem:2.1}
Let $R$ be a ring. Assume that $D$ is a division subring of $R$ and $M$ is a subgroup of $R^*$ normalizing $D^*$. If $x\in D^*$ is algebraic over $K:=C_D(M)$ with the minimal polynomial $f(t)\in K[t]$, then a polynomial $h(t)\in D[t]$ vanishes on $x^M$ if and only if $h(t)\in D[t] f(t)$.
\end{lemma}
\begin{proof} Note that $K=C_D(M):=\{d\in D~\vert~ dm=md, \forall m\in M\}$ may not be a field. We claim that if $h(t)\in D[t]\setminus\{0\}$ vanishes on $x^M$ then ${\rm deg} h\geq {\rm deg} f$. Assume that this conclusion is false. Then, we can take a polynomial 
$$h(t)=t^m+d_1t^{m-1}+\cdots+d_m$$
with the smallest $m<{\rm deg} f$ such that $h(x^M)=0$. Clearly, we have $h(t)\not\in K[t]$. So, there exists some $d_i\not\in K$, and we can pick an element $e\in M$ such that $ed_i\neq d_ie$. Since $M$ normalizes $D^*$, for any $b\in D$, we have $b':=ebe^{-1}\in D$. For any $a\in x^M$, we can conjugate the equation
$$a^m+d_1a^{m-1}+\cdots+d_m=0$$
by element $e$ to get
$$(a')^m+d'_1(a')^{m-1}+\cdots +d'_m=0.$$
On the other hand, we also have
$$(a')^m+d_1(a')^{m-1}+\cdots + d_m=0.$$
It follows that the nonzero polynomial $\sum_{j=1}^m (d_j-d'_j)t^{m-j}$ vanishes on $ex^Me^{-1}=x^M$, and  its degree is less than $m$, a contradiction. Since the coefficients of $f(t)$ commute with elements from $M$, it is easy to see that $h(t)$ vanishes on $x^M$ for all $h(t)\in D[t]f(t)$. 

Conversely, assume that $h(t)\in D[t]\setminus \{0\}$ and $h(x^M)=0$. By the division algorithm, we can write $h(t)=q(t)f(t)+r(t)$, where $r(t)=0$ or ${\rm deg} r<{\rm deg} f$. Since $h(x^M)=0$ and $f(x^M)=0$, it follows that $r(x^M)=0$. By the claim above, we have $r(t)=0$. So, $h(t)\in D[t]f(t)$, and the proof is now complete.
\end{proof}
Now, using Lemma \ref{lem:2.1}, we get easily the following theorem we need in the next study.
 
\begin{theorem}\label{thm:2.2}
Let $R$ be a ring. Assume that  $D$ is a division subring of $R$ and $M$ is a subgroup of $R^*$ normalizing $D^*$. If $x\in D^*$ is algebraic over $K:=C_D(M)$ with the minimal polynomial $f(t)$ of degree n, then there exist $x_1,\ldots,x_{n-1}\in x^{MD^*}$ such that $$f(t)=(t-x_{n-1})\cdots (t-x_1)(t-x)\in D[t].$$
\end{theorem}
\begin{proof} Take a factorization 
$$f(t)=g(t)(t-x_r)\cdots (t-x_1)(t-x)$$
with $g(t)\in D[t]$, $x_1,\ldots ,x_r\in x^{MD^*}$, where $r$ is chosen as large as possible. We claim that $h(t):=(t-x_r)\cdots (t-x_1)(t-x)$ vanishes on $x^M$. Indeed, consider an arbitrary element $y\in x^M$. If $h(y)\neq 0$, then by \cite[(16.3), p. 263]{lam}, $g(x_{r+1})=0$, where $x_{r+1}=aya^{-1}\in x^{MD^*}, a=h(y)$. It follows that $g(t)=g_1(t)(t-x_{r+1})$ for some $g_1(t)\in D[t]$, and so 
$$f(t)=g_1(t)(t-x_{r+1})(t-x_r)\cdots (t-x_1)(t-x).$$
Since this contradicts to  the choice of $r$, we must have  $h(x^M)=0$.  So, in view of  Lemma~\ref{lem:2.1}, $r=n-1$. Hence,  $f(t)=(t-x_{n-1})\cdots (t-x_1)(t-x)$ as  required.
\end{proof}
We notice that by taking $R=D$ and $M=D^*$ in Theorem~\ref{thm:2.2}, we get Wedderburn's factorization theorem. 

In view of   Theorem \ref{thm:2.2}  and \cite[(16.3), p. 263]{lam},  the following corollary is immediate.

\begin{corollary}\label{cor:2.3}
Let $R$ be a ring. Assume that  $D$ is a division subring of $R$ and $M$ is a subgroup of $R^*$ normalizing $D^*$. If $x\in D^*$ is algebraic over $K:=C_D(M)$ with the minimal polynomial $f(t)$ and $y$ is a  root of $f(t)$ in $D$, then $y\in x^{MD^*}$.
\end{corollary}
\begin{corollary}\label{cor:2.4}
Let $R$ be a ring. Assume that $D$ is a division subring of $R$, and  $M$ is a subgroup of $R^*$ such that $D^*\trianglelefteq M$. If $x\in D^*$ is algebraic over $K:=C_D(M)$ with the minimal polynomial $f(t)$ of degree n, then $K$ is contained in the center of $D$ and there exists an element $c_x\in [M,x]\cap K(x)$ such that $x^n=N_{K(x)/K}(x)c_x$ with $N_{K(x)/K}(c_x)=1$, where $N_{K(x)/K}$ is the norm of $K(x)$ to $K$.
\end{corollary}
\begin{proof} Since $D^*\leq M$, $K$ is contained in the center of $D$ and $K(x)$ is a field. If $b=N_{K(x)/K}(x)$, then by Theorem \ref{thm:2.2}, we have $b=x^{r_1}\cdots x^{r_n}$ with $r_1,\ldots ,r_n\in M$. We can write $b$ in the following form:
$$b=[r_1,x][r_2,x]^x[r_3,x]^{x^2}\cdots [r_n,x]^{x^{n-1}}x^n.$$
Putting 
$$c^{-1}_x=[r_1,x][r_2^x,x][r_3^{x^2},x]\cdots [r_n^{x^{n-1}},x],$$
we have $c_x=b^{-1}x^n\in [M,x]\cap K(x)$. So, 
$$N_{K(x)/K}(c_x)=N_{K(x)/K}(b^{-1})N_{K(x)/K}(x)^n=b^{-n}b^n=1.$$
\end{proof}
This corollary can be reformulated in the following form which should be convenient  in some cases of its application. In particular, we shall use it in the proof of Theorem \ref{thm:3.6} in the next section. 

\begin{corollary}\label{cor:2.5}
Let $R$ be a ring. Assume that  $D$ is a division subring of $R$ and $M$ is a subgroup of $R^*$ normalizing $D^*$. If $x\in Z(D)^*\cap M$ is algebraic over $K:=C_D(M)$ with the minimal polynomial $f(t)$ of degree n, then there exists an element $c_x\in [M,x]\cap K(x)$ such that $x^n=N_{K(x)/K}(x)c_x$ with $N_{K(x)/K}(c_x)=1$.
\end{corollary}
\begin{proof} Let $b=N_{K(x)/K}(x)\in K$. By Theorem \ref{thm:2.2}, we have $b=x^{r_1d_1}\cdots x^{r_nd_n}$, with $r_1,\ldots ,r_n\in M$ and $d_1,\ldots ,d_n\in D^*$. We can write $b$ in the following form:
$$b=[r_1d_1,x][r_2d_2,x]^x[r_3d_3,x]^{x^2}\cdots [r_nd_n,x]^{x^{n-1}}x^n.$$
Since $x\in Z(D)^*$, we get
$$b=[r_1,x][r_2,x]^x[r_3,x]^{x^2}\cdots [r_n,x]^{x^{n-1}}x^n.$$
Putting 
$$c^{-1}_x=[r_1,x][r_2^x,x][r_3^{x^2},x]\cdots [r_n^{x^{n-1}},x],$$ 
we have $c_x=b^{-1}x^n\in [M,x]\cap K(x)$. So, 
$$N_{K(x)/K}(c_x)=N_{K(x)/K}(b^{-1})N_{K(x)/K}(x)^n=b^{-n}b^n=1.$$
\end{proof}

\section{Maximal subgroups of subnormal subgroups}

In this section, we describe the structure of maximal subgroups in an arbitrary subnormal subgroup of $D^*$, and we show their influence to the whole structure of $D$. In the first, we prove the  following useful lemmas we need for our further study.

\begin{lemma}\label{lem:3.1}
If $D $ is a division ring with center $F$ and $G$ is a soluble-by-locally finite subnormal subgroup of $D^*$, then $G\subseteq F$.
\end{lemma}
\begin{proof} By assumption, there is a soluble normal subgroup $H$ of $G$ such that $G/H$ is locally finite. Consequently, $H$ is a soluble subnormal subgroup of $D^*$, and in view of \cite[14.4.4, p. 440]{scott2},  we have $H\subseteq F$. It follows that  $G/Z(G)$ is locally finite. Now, by \cite[Lemma~3]{maxn}, $G'$ is locally finite too. Therefore, $G'$ is a torsion subnormal subgroup of $D^*$, and by \cite[Theorem 8]{her3},  $G'\subseteq F$. Hence, $G$ is soluble, and again by \cite[14.4.4, p. 440]{scott2},  $G\subseteq F$.
\end{proof}

\begin{lemma}\label{lem:3.1a}
Let $D$ be a division ring with center $F$ and $G$ be a subnormal subgroup of $D^*$. If $G$ is locally polycyclic-by-finite (e.g. if $G$ is locally nilpotent), then $G\subseteq F$.
\end{lemma}
\begin{proof} Assume that $G$ is not contained in $F$. Then, by Stuth's theorem (see, for example, \cite[14.3.8, p. 439]{scott2}),  we have $D=F(G)$. By a result of Lichtman (see \cite[4.5.2, p.~155]{wehrfritz}),  together with an exercise in \cite[p. 162] {wehrfritz}, or with the fact that $G$ is contained in a unique maximal locally polycyclic-by-finite normal subgroup of $D^*$, it follows that $G$ contains a non-cyclic free subgroup. But this contradicts to the fact that $G$ is locally polycyclic-by-finite.
\end{proof}

\begin{theorem}\label{thm:3.3}
Let $D$ be a division ring with center $F$, and assume that $G$ is a non-central subnormal subgroup of $D^*$. If $x$ is a non-central element of $D$, then $|x^G|=|D|$. 
\end{theorem}
\begin{proof} By \cite[14.4.3, p. 439]{scott2}, $D$ is generated by $x^G$. Recall that if a ring $R$ has an infinite generating set $S$, then $R$ and $S$ have the same cardinality. Therefore,  to prove the theorem, it suffices to show that $x^G$ is infinite for any $x\in D\setminus F$. Thus, assume that $|x^G|<\infty$, or, equivalently, that $[G:C_G(x)]<\infty$. Since $G$ normalizes $F(x^G)$, by Stuth's theorem,  $F(x^G) =D$. Putting $H=(C_G(x))_G$, we have  $H\subseteq C_D(x^G)=C_D(D)=F$. The condition  $[G:C_G(x)]<\infty$ implies $[G:H]<\infty$. Hence, by Lemma \ref{lem:3.1}, $G\subseteq F$, a contradiction. Therefore, $|x^G|$ is infinite as we desired to prove. 
\end{proof}

 Theorem \ref{thm:3.3} shows that  non-central subnormal subgroups of $D^*$ are ``big".  In the following, using this fact, we show the role of  $FC$-elements in maximal subgroups of a subnormal subgroup of $D^*$. The results obtained in the next theorem will be used as principal tools for our study in the remaining part of this paper.

\begin{theorem}\label{thm:3.4}
Let $D$ be a division ring with center $F$ and $G$ be a subnormal subgroup of $D^*$. Assume that $M$ is  a maximal subgroup of $G$ containing a non-central $FC$-element $\alpha$. If $K=F(\alpha^M)$ and $H=C_D(K)$, then the following conditions hold:

(i) $K$ is a field, $[K:F]=[D:H]_r<\infty$ and $F(M)=D$, where $[D:H]_r$ is the right dimension of $D$ over $H$.

(ii)  $K^*\cap G$ is the $FC$-center,  and also, it is the Fitting subgroup of $M$.

(iii)  The field extension $K/F$ is Galois, $H^*\cap G$ is normal in $M$, and  $M/H^*\cap G\cong \mathrm{Gal}(K/F)$. Moreover, $\mathrm{Gal}(K/F)$ is a finite simple group.

(iv)  If $H^*\cap G\subseteq K$, then $H=K$ and $[D:F]<\infty$.
\end{theorem}

\begin{proof} Firstly, we claim that $F(M)=D$. Since $M$ is a maximal subgroup of $G$, either $F(M)^*\cap G=M$ or 
$F(M)^*\cap G=G$. If $F(M)^*\cap G=M$, then $M$ is a subnormal subgroup of $F(M)^*$ containing a non-central $FC$-element $\alpha$. But, this is impossible in view of  Theorem \ref{thm:3.3}. Thus, $F(M)^*\cap G=G$, and by Stuth's theorem,  $F(M)=D$. Since $\alpha$ is an $FC$-element of $M$, $[M:C_M(\alpha)]<\infty$. Setting  $N=(C_M(\alpha))_M$, $K=F(\alpha^M)$ and $H=C_D(K)$, we have $N\vartriangleleft M$, $N\leq H^*$, and $[M:N]$ is finite. Clearly,  $M$ normalizes $K^*$. Therefore, by the maximality of $M$ in $G$, it follows that  either $G$ normalizes $K^*$ or $K^*\cap G\leq M$. If $G$ normalizes $K^*$, then by Stuth's theorem,  $K=D$, so $H=F$.  This implies $N\leq F^*\cap M\leq Z(M)$, and, consequently, $[M:Z(M)]<\infty$. Now, in view of  \cite[Theorem 1]{kiani-ram}, $M$ is abelian, a contradiction.  Thus, $K^*\cap G=K^*\cap M\trianglelefteq M$. Hence, $K^*\cap G$ is a subnormal subgroup of $K^*$ containing the set $\alpha^M$ of  $FC$-elements in $M$. By Theorem~\ref{thm:3.3},  $\alpha^M\subseteq Z(K)$; consequently, $K$ is a field. Since $H=C_D(K)$ and $M$ normalizes $K^*$, it follows that $M$ also normalizes $H^*$. By the maximality of $M$ in $G$, either $G$ normalizes $H^*$ or $H^*\cap G\leq M$. If $G$ normalizes $H^*$, then by Stuth's theorem,  either $H=D$ or $H\subseteq F$. However, both these cases are impossible since $K\subseteq C_D(K)=H$, and $K$ contains an element $\alpha\not\in F=Z(D)$. Therefore, $H^*\cap G=H^*\cap M\trianglelefteq M$. The conditions $N\leq H^*\cap G$ and $[M:N]<\infty$ imply $[M:H^*\cap G]<\infty$, and $M=\bigcup_{i=1}^tx_i(H^*\cap G)$ for some transversal $\{x_1, \ldots, x_t\}$ of $H^*\cap G$ in $M$. Since $M$ normalizes $H^*$, it is easy to see that $R:=\sum_{i=1}^tx_iH$ is a ring. Also, $R$ is a finite-dimensional right vector space over its division  subring $H$. It follows that $R$ is a division subring of $D$ containing $F(M)$. Therefore $R=D$, and $[D:H]_r<\infty$. Using Double Centralizer theorem \cite[15.4, p. 253]{lam},  we get $[K:F]=[D:H]_r<\infty$, and $Z(H)=C_D(H)=K$. Thus, $(i)$ is established.
\\\indent
For any $a\in M$, the mapping $\theta_a:K\to K$ given by $\theta_a(x)=axa^{-1}$ is an $F$-automorphism of $K$. Now, consider  the mapping $\psi:M\to \mathrm{Gal}(K/F)$ given by $\psi(a)=\theta_a$. Clearly, $\psi$ is a group homomorphism with $\mathrm{ker}\psi=C_M(K^*)=C_D(K)^*\cap M=H^*\cap M=H^*\cap G$. The condition $F(M)=D$ implies $C_D(M)=F$. Therefore, the fixed field of $\psi (M)$ is $F$. By Galois correspondence, we conclude that $\psi$ is a surjective homomorphism, and $K/F$ is a Galois extension.  Hence, $M/H^*\cap G\cong\mathrm{Gal}(K/F)$ is a finite group. Assume that $\mathrm{Gal}(K/F)$ is not simple. Then, there exists some subfield $L$ of $K$ containing $F$ such that $\theta(L)=L$ for all $\theta\in \mathrm{Gal}(K/F)$. Thus, $L^*$ and $C_D(L)^*$ are normalized by $M$,  $E:=C_D(L)\neq D$ and $E\not\subseteq F$. So, by Stuth's theorem, $E^*$ is not normalized by $G$. If $E^*\cap G\not\subseteq M$, then $G=M(E^*\cap G)$ normalizes $E^*$, a contradiction. Thus, $E^*\cap G\leq M$. Since $K^*\cap G\trianglelefteq M$ and $K^*\cap G\leq E^*\cap G$, we have $K^*\cap G\trianglelefteq E^*\cap G$. So, $K^*\cap G$ is an abelian subnormal subgroup of $E^*$. By Lemma \ref{lem:3.1}, it follows that  $K^*\cap G\subseteq Z(E)= L$. Hence, $K=F(\alpha^M)\subseteq F(K^*\cap G)\subseteq L$, a contradiction. Therefore, $\mathrm{Gal}(K/F)$ is simple, and the proof of $(iii)$ is now complete.
\\\indent
If $H^*\cap G\subseteq K$, then $H^*\cap G=K^*\cap G$. The condition $[M: H^*\cap G]< \infty$ implies $[M:K^*\cap G]<\infty$. Suppose that $\{y_1,\ldots,y_l\}$ is a transversal of $K^*\cap G$ in $M$. Since $M$ normalizes $K^*$,  $Q=\sum_{i=1}^ly_iK$  is a division ring. Clearly, $Q$ contains both $F$ and $M$. So, in view of $(i)$, we have $Q=D$.  It is easy to see that $[D:K]_r=[Q:K]_r\leq|M/H^*\cap G|=|\mathrm{Gal}(K/F)|=[K:F]$. Hence, by Double Centralizer theorem,  $K$ is a maximal subfield of $D$, $K=H$, and $[D:F]<\infty$. Thus $(iv)$ is established.
\\\indent
To prove $(ii)$, firstly, we claim that $K^*\cap G$ is a maximal abelian normal subgroup of $M$. Indeed, assume that $C$ is a maximal abelian normal subgroup of $M$ containing $K^*\cap G$. Then, $\alpha^M\subseteq C$, and it follows that $C\leq H^*\cap G\leq M$. Consequently, $C\trianglelefteq H^*\cap G$. Hence, $C$ is an abelian subnormal subgroup of $H^*$. Now, by Lemma~ \ref{lem:3.1}, $C\subseteq Z(H)=K$, and this implies $K^*\cap G=C$.
\\\indent
To prove that $K^*\cap G$ is the Fitting subgroup of $M$, it suffices to show that $K^*\cap G$ is a maximal nilpotent normal subgroup of $M$. 
Thus, assume that $A$ is a nilpotent normal subgroup  of $M$ which strictly contains $K^*\cap G$. Then, $B=H^*\cap G\cap A$ is a nilpotent subnormal subgroup of $H^*$. Hence, by Lemma \ref{lem:3.1}, we conclude that $B\subseteq Z(H)=K$, and, consequently, $B=K^*\cap G\cap A$. If $A\subseteq H^*\cap G$, then $A=B\subseteq K^*\cap G$, a contradiction. Therefore, $A\not\subseteq H^*\cap G$, and it follows that $A(H^*\cap G)/H^*\cap G$ is a  nontrivial normal subgroup of $M/H^*\cap G$. The simplicity of the group $M/H^*\cap G$ implies $M/H^*\cap G=A(H^*\cap G)/H^*\cap G\cong A/B$. Suppose that $S=\sum_{i=1}^mz_iK$, where $\{z_1,\ldots ,z_m\}$ is a transversal of $B$ in $A$. Since $A$ normalizes $K^*$ and $B\subseteq K$, $S$ is a division ring and $[S:K]_r\leq m$. Recall that $M$ normalizes $A$ and $K$. So, $M$ also normalizes $S$.  By the maximality of $M$ in $G$,  either $G$ normalizes $S$ or $S^*\cap G\leq M$. If the second case occurs, then $A$ is a nilpotent subnormal subgroup of $S^*$. By Lemma~\ref{lem:3.1}, $A$ is abelian, and  this contradicts to the fact that $K^*\cap G$ is a maximal abelian normal subgroup of $M$. Thus, $G$ normalizes $S$, and by Stuth's theorem, $S=D$. Therefore,  $[D:K]_r\leq m=|M/H^*\cap G|=|\mathrm{Gal}(K/F)|=[K:F]$. This  implies  $[D:F]=m^2$, and $K=H$ is a maximal subfield of $D$. From the fact that $M/H^*\cap G=A(H^*\cap G)/H^*\cap G$, and $K^*\cap G<A$, it follows that $M=A$. Since $[D:F]<\infty$, $M$ can be considered as a nilpotent linear group no containing unipotent elements ($\neq 1$). By a result in \cite[p. 114]{dixon},  $[M:Z(M)]$ is finite, which contradicts to \cite[Theorem 1]{kiani-ram}. Hence, $K^*\cap G$ is the Fitting subgroup of $M$.
\\\indent
For any $x\in K^*\cap G$, the elements of $x^M\subseteq K$ have the same minimal polynomial over $F$, so $|x^M|<\infty$. Now, assume that $x\in M\setminus K^*\cap G$. If $|x^M|<\infty$, then by what we have proved, $F(x^M)\cap G$ is the Fitting subgroup of $M$ which is different from  $K^*\cap G$, a contradiction.  Hence, $|x^M|=\infty$ and $(ii)$ follows. Thus, the proof of the theorem is now complete.
\end{proof}

From Theorem \ref{thm:3.4}, we get the following  corollary which is   convenient for further applications.

\begin{corollary}\label{cor:3.5}
Let $D$ be a division ring with center $F$ and $G$ be a subnormal subgroup of $D^*$. Assume that $M$ is a maximal subgroup of G. If $M$ contains an abelian normal subgroup $A$ and an element $\alpha\in A\setminus Z(M)$ which is algebraic over $F(Z(M))$, then $K=F(A)$, and $H=C_D(K)$ satisfy the conditions (i) - (iv) of Theorem \ref{thm:3.4}. Moreover, $F(A)=F[A]$.
\end{corollary}
\begin{proof} Since $A\trianglelefteq M$, the elements of $\alpha^M$ in the field $F(Z(M)A)$ have the same minimal polynomial over $F(Z(M))$. Hence, $|\alpha^M|<\infty$. Now, it is clear that  $K=F(\alpha^M)$ and $H=C_D(K)$  satisfy the conditions (i) - (iv) of Theorem \ref{thm:3.4}. Since $[K:F]<\infty$, by (ii), $A\subseteq K$, and $K=F(A)=F[A]$.
\end{proof}

\begin{theorem}\label{thm:3.6}
Let $D$ be a division ring with center $F$ and $G$ be a subnormal subgroup of $D^*$. If  $M$ is  a non-abelian metabelian maximal subgroup of $G$, then the following conditions hold:

(i) There exists a maximal subfield $K$ of $D$ such that $K/F$ is a finite Galois extension with $\mathrm{Gal}(K/F)\cong M/K^*\cap G\cong \mathbb{Z}_p$ for  some prime $p$, and $[D:F]=p^2$. 

(ii) The subgroup $K^*\cap G$ is the $FC$-center. Also, $K^*\cap G$ is  the Fitting subgroup of $M$, and  $M/M'Z(M)\cong\bigoplus_{i\in I}\mathbb{Z}_p$. Furthermore, for any $x\in M\setminus K$, we have $x^p\in F$ and $D=F[M]=\bigoplus_{i=1}^pKx^i$.
\end{theorem}
\begin{proof} 

(i) Denote  a maximal abelian normal subgroup of $M$ containing $M'$ by $A$ and consider an arbitrary subgroup $N$ of $M$ which properly contains  $A$. Since  $M'\leq N$,  $N$ is a normal subgroup of $M$. The maximality of $M$ in $G$ implies that either $G$ normalizes $F(N)^*$ or $F(N)^*\cap G\leq M$. If the second case occurs, then $N\trianglelefteq F(N)^*\cap G$. So, $N$ is a metabelian subnormal subgroup of $F(N)^*$. By Lemma~\ref{lem:3.1}, $N$ is abelian, which contradicts to the maximality of $A$. Hence,  $G$ normalizes $F(N)^*$, and by Stuth's theorem,  $F(N)=D$.
\\\indent
Now, let $a$ be an element from $M\setminus A$, and assume that $a$ is transcendental over $F(A)$. Set $T=F(A)^*\langle a^2\rangle$. Since $a$ normalizes $F(A)^*$, it is not hard to see that $F[T]=\bigoplus_{i\in\mathbb{Z}}F(A)a^{2i}$, and $(F[T], F(A), T, T/F(A)^*)$ is a crossed product. Since $T/F(A)^*\cong\langle a^2\rangle$ is an infinite cyclic group,  by \cite[1.4.3, p. 26]{wehrfritz},  $F[T]$ is an Ore domain. On the other hand, by what we have proved before, we have $F(T)=D$. Therefore,  there exist two elements $s_1, s_2\in F[T]$ such that $a=s_1s_2^{-1}$. Writing   $s_1=\sum_{i=l}^mk_ia^{2i}$ and $s_2=\sum_{i=l}^mk'_ia^{2i}$, with $k_i,\ k'_i\in F(A)$ for any $l\leq i\leq m$, we have  $\sum_{i=l}^mak'_ia^{2i}=\sum_{i=l}^mk_ia^{2i}$. If for any $l\leq i\leq m$ we set $l_i=ak'_ia^{-1}$, then $l_i$'s are elements of $F(A)$, and  $\sum_{i=l}^ml_ia^{2i+1}=\sum_{i=l}^mk'_ia^{2i}$. This shows that $a$ is algebraic over $F(A)$, say of degree $n$. Using the fact that $a$ normalizes $F(A)^*$, we see that $R=\bigoplus_{i=0}^{n-1}F(A)a^i$ is a domain which is a finite-dimensional left vector space over $F(A)$. Therefore, $R$ is a division ring, and $R=F(A\langle a\rangle)$. By what we have  proved before, we conclude that $R=D$. This means that $[D:F(A)]_l<\infty$. Now, by Double Centralizer theorem, $[D:F]<\infty$.
\\\indent
Assume that $A$ is contained in $Z(M)$. From the condition $M'\leq A$, it follows that $\langle A,x\rangle$ is an abelian normal subgroup of $M$ properly containing $A$ for any  $x\in M\setminus A$. But, this contradicts to the maximality of $A$. Hence, $A\not\subseteq Z(M)$. Since $[D:F]<\infty$, all elements of $A\setminus Z(M)$ are algebraic over $F$. In view of  Corollary \ref{cor:3.5}, there exists a subfield $K$ of $D$ such that $K$ and  $H=C_D(K)$ satisfy the conditions $(i)-(iv)$ of Theorem~\ref{thm:3.4}. So, $H^*\cap G$ is a metabelian subnormal subgroup of $H^*$. By Lemma~ \ref{lem:3.1}, $H^*\cap G\subseteq Z(H)=K$. The condition $(iv)$ implies that $K=H$ is a maximal subfield of $D$. Since $M$ is metabelian, $M/K^*\cap G$ is simple and metabelian. We conclude that $\mathrm{Gal}(K/F)\cong M/K^*\cap G\cong \mathbb{Z}_p$, where $p$ is a prime number,  and  $[D:F]=p^2$. 

(ii) Since $[M:K^*\cap G]=p$ and $D$ is algebraic over $F$,  $D=F(M)=F[M]=\sum_{i=1}^pKx^i$ for any $x\in M\setminus K$. Therefore, $D=\bigoplus_{i=1}^pKx^i$. Suppose that  $x^p\not\in F$. Then, $Z(M)=M\cap F^*$ because  $F(M)=D$. Therefore, $C_M(x^p)\neq M$. On the other hand, since $\langle x,K^*\cap G\rangle\leq C_M(x^p)$, and $[M:K^*\cap G]$ is prime,  we get  $C_M(x^p)=M$, a contradiction. Thus, $x^p\in F$. Now, by Corollary \ref{cor:2.5},  for any $y\in K^*\cap M=K^*\cap G$, we have $y^p\in M'F^*$. Hence, $y^p\in M'(M\cap F^*)=M'Z(M)$. So, $M/M'Z(M)$ is an abelian group of exponent $p$, and by Bear-Prufer's theorem \cite[p. 105]{rob},  $M/M'Z(M)\cong\bigoplus_{i\in I}\mathbb{Z}_p$. This completes the proof.
\end{proof}

\begin{lemma}\label{lem:3.8}
If $D$ is a centrally finite  division ring with center $F$, then $D'\cap F^*$ is finite.
\end{lemma}
\begin{proof} Suppose $[D:F]=n^2$. By taking the reduced norm, we obtain $x^n=1$ for all $x\in D'\cap F^*$. Since $F$ is a field, $D'\cap F^*$ is finite.
\end{proof}

\begin{lemma}\label{lem:3.9}
Let $D$ be a division ring with center $F$ and $G$ be a subnormal subgroup of $D^*$. If  $M$ is a non-abelian maximal subgroup of $G$ such that $M'$ is locally finite, then $M/Z(M)$ is locally finite, $M'$ is locally cyclic, and the conclusions of Theorem \ref{thm:3.6} follow.
\end{lemma}
\begin{proof} Assume that  $M'$ is non-abelian. In view of  Wedderburn's Little theorem, we must have ${\rm Char} D$=0. First, we claim that there exists a torsion abelian normal subgroup of $M$ which is not contained in $Z(M)$. By the maximality of $M$ in $G$, either $F(M')\cap G\subseteq M$ or $G$ normalizes $F(M')$. If $F(M')\cap G\subseteq M$, then $M'$ is a locally finite subnormal subgroup of $F(M')^*$. In view of  Lemma \ref{lem:3.1}, $M'$ is abelian, a contradiction. Therefore, $G$ normalizes $F(M')$, and  by Stuth's theorem,  $F[M']=F(M')=D$. We notice that by \cite[2.5.5, p. 74]{wehrfritz},  there exists a metabelian normal subgroup of $M'$ of finite index $n$. By setting $Q=\langle\{ x^n|x\in M'\}\rangle$, we see that $Q$ is a metabelian normal subgroup of $M$ and $Q'$ is a torsion abelian normal subgroup of $M$. If $Q'$ is not contained in $Z(M)$, then we are done. Hence, we may assume that $Q'\leq Z(M)$. So, $Q$ is locally finite and nilpotent. By \cite[2.5.2, p. 73]{wehrfritz}, $Q$ contains an abelian subgroup $B$ such that $[Q:B]<\infty$. Clearly,  $C=\cap_{x\in M} xBx^{-1}$ is a torsion abelian normal subgroup of $M$, and we may assume that $C\leq Z(M)$. On the other hand, since $[Q:B]<\infty$ and $Q\trianglelefteq M$,  $Q/C$ has a bounded exponent. From the condition  $C\leq Z(M)$ and the definition of $Q$, we see that the exponent of $M'/Z(M')$ is also bounded. Therefore, by the classification theorem of locally finite groups in division ring \cite[2.5.9, p. 75]{wehrfritz}, we conclude that $M'$ is abelian-by-finite. Now, let $A$ be an abelian normal subgroup of $M'$ of finite index. Since $F[M']=D$,  $D$ is a finite-dimensional vector space over  $F[A]$. By Double Centralizer theorem, $[D:F]<\infty$. So, by Lemma \ref{lem:3.8}, $|A\cap F^*|<\infty$. The condition $F[M']=D$ implies $Z(M')=F\cap M'$. Hence, $A/A\cap F^*$ has a bounded exponent, so is $A$. Considering  $A$  as the multiplicative subgroup of a field, we conclude  that $A$ is finite, so is $M'$. It follows that $M$ is an $FC$-group.  By Theorem \ref{thm:3.4}, $M$ is abelian,  a contradiction. Therefore, there exists a torsion abelian normal subgroup of $M$ which is not contained in $Z(M)$. 
\\\indent
Now, by Corollary \ref{cor:3.5}, there exists  a field $K$ such that $K$ and $H=C_D(K)$ satisfy the conditions $(i) - (iv)$ of Theorem \ref{thm:3.4}. If  $N=M'\cap H^*$, then  $N\trianglelefteq M\cap H^*=G\cap H^*$. Therefore, $N$ is a locally finite subnormal subgroup of $H^*$. By Lemma \ref{lem:3.1}, $N=M'\cap H^*$ is abelian; consequently, $M'\not\subseteq M\cap H^*$. If $M/M\cap H^*$ is abelian, then $M'\subseteq M\cap H^*$, a contradiction. Thus, from $(iii)$ of Theorem \ref{thm:3.4}, $M'/N\cong M'(M\cap H)/M\cap H=M/M\cap H$ is a non-abelian finite simple group. Assume that $\{x_1,\ldots,x_m\}$ is a transversal of $N$ in $M'$. Then, $S=\langle x_1,\ldots,x_m\rangle$ is a finite subgroup of $M'$, and
$S/S\cap N\cong SN/N=M'/N.$ Since $S$ has a quotient group which is a finite simple non-abelian group, $S$ is unsoluble. The classification theorem of finite groups in division rings \cite[2.1.4, p. 46]{wehrfritz} states that the only unsoluble finite subgroup of a division ring is $SL(2,5)$. Hence, $M'=S\cong SL(2,5)$ is a finite group. Thus, $M$ is an $FC$-group, and  by Theorem \ref{thm:3.4}, it is abelian, a contradiction. Therefore, $M'$ is abelian, and it can be considered as a torsion multiplicative subgroup of a field. Hence, $M'$ is locally cyclic, and the proof is now  finished by  Theorem~ \ref{thm:3.6}.
\end{proof}

In \cite{free}, M. Mahdavi-Hezavehi studied the existence of non-cyclic free subgroups in a maximal subgroup of a centrally finite division ring $D$. The main result in this work asserts that if $D$ is a non-crossed product, then every maximal subgroup of $D$ contains a non-cyclic free subgroup. Here, we study the same problem for maximal subgroups in any subnormal subgroup $G$ of a locally finite division ring $D$. The result we get in the following theorem shows that the missing of non-cyclic free subgroups in a maximal subgroup of $G$ entails $D$ to be centrally finite. Moreover, if $G=D^*$, then $D$ must be a crossed product. 
 
\begin{theorem}\label{thm:3.15}
Let $D$ be a locally finite division ring with center $F$ and $G$ be a subnormal subgroup of $D^*$. Assume that $M$ is  a non-abelian maximal subgroup of $G$. If $M$ contains no non-cyclic free subgroups, then $[D:F]<\infty$, $F(M)=D$, and there exists a maximal subfield $K$ of $D$ such that $K/F$ is a Galois extension, $\mathrm{Gal}(K/F)\cong M/K^*\cap G$ is a finite simple group, and $K^*\cap G$ is the $FC$-center. Also, $K^*\cap G$ is  the Fitting subgroup of $M$.
\end{theorem}
\begin{proof} Firstly, we show that $F(M)=D$. Indeed, if $F(M)\neq D$, then by Stuth's theorem, $F(M)$ is not normalized by $G$. Since $M$ is a maximal of $G$, $F(M)\cap G=M$. Therefore,  $M$ is a non-abelian subnormal subgroup of $F(M)^*$  containing no a non-cyclic free subgroup, and this contradicts to \cite[Theorem 11]{hai-ngoc}. Thus, $F(M)=D$. This implies $Z(M)=M\cap F^*$. Now,  by \cite[Theorem 5]{hai-ngoc},  there exists a locally  abelian normal subgroup $A$ of $M$ such that $M/A$ is  finite. If $A\subseteq Z(M)$, then $M/Z(M)$ is finite, so is $M'$, and  the result follows from  Lemma~\ref{lem:3.9}. Assume that $A\not\subseteq Z(M)$. By  Corollary \ref{cor:3.5},  there exists  a subfield $K$ of $D$ such that $H=C_D(K)$, and $K$ satisfy the conditions $(i)-(iv)$ of Theorem \ref{thm:3.4}. Therefore, $H^*\cap G$ is a subgroup of $M$,  and clearly, it is  a subnormal subgroup of $H^*$. The condition  $H=C_D(K)$  implies $K\subseteq Z(H)$, where $Z(H)$ is the center of $H$. On the other hand, since $Z(H)=C_H(H)\subseteq C_D(H)=C_D(C_D(K))=K,$ it follows that $Z(H)=K$. Hence, by \cite[Theorem 11]{hai-ngoc},  $H^*\cap G\subseteq K$. The proof is now finished in view of  the condition $(iv)$ of Theorem \ref{thm:3.4}.
\end{proof}

\section{Maximal subgroups of ${\rm GL}_1(D)$}

In this section, we prove more precise results in the case $G=D^*$. 
 
\begin{theorem}\label{thm:4.1}
Let $D$ be a division ring with center $F$ and assume that $M$ is a maximal subgroup of $D^*$ containing a non-central $FC$-element $\alpha$. By setting $K=F(\alpha^M)$, the following conditions hold:

(i)  $K$ is a field, $[K:F]<\infty$, and $F[M]=D$.

(ii)  $K^*$ is the Hirsch-Plotkin radical of $M$.

(iii)  $F^* < K^* \leq C_D(K)^* < M < D^*$.

(iv)  $K/F$ is a Galois extension, $M/C_D(K)^* \cong \mathrm{Gal}(K/F)$ is a finite simple group.

(v)  For any $ x \in K,  |x^M|\leq [K:F]$,  and for any $ y \in D\setminus K,  |y^M|=|D|$.
\end{theorem}
\begin{proof} We notice that $K$ and $H=C_D(K)$ satisfy the conditions $(i)-(iii)$ of Theorem~\ref{thm:3.4} (for $G=D^*$). As in the first paragraph of the proof of  Theorem \ref{thm:3.4}, we can write $D=\sum_{i=1}^tx_iH$, where $x_1,\ldots,x_t\in M$. In this case, we have $H^*\leq M$, and this  implies $F[M]=D$. Let $x$ be an element of $K^*$. Since $K$ is algebraic over $F$ and $K^*\vartriangleleft M$,  $x^M\subseteq K$ and the elements of $x^M$ have the same minimal polynomial $f$ over $F$. So, $|x^M|\leq \text{\rm deg}f\leq [K:F]$. For any $y\in D\setminus K$,  $C_H(y)$ is a proper division subring of $H$ since  $y\not\in K=C_D(H)$. By \cite[14.2.1, p. 429]{scott2} and the fact that $[D:H]_r<\infty$, we have $|y^M|=[M:C_M(y)]\geq [H^*:C_H(y)^*]=|H|=|D|$. Thus,  $(i), (iii), (iv)$ and $(v)$ hold.
\\\indent
It remains to  prove $(ii)$. Let $A$ be the Hirsch-Plotkin radical of $M$, and suppose  that $K^*<A$. Then, $B=H^*\cap A$ is a locally nilpotent normal subgroup of $H^*$. Hence, by Lemma \ref{lem:3.1a},  $B\subseteq Z(H)=K$, and consequently, $B=K^*$. If $A\subseteq H^*$, then $A=B=K^*$, a contradiction. Therefore, $A\not\subseteq H^*$. Thus, $AH^*/H^*$ is a  nontrivial normal subgroup of $M/H^*$. The simplicity of the group  $M/H^*$ implies $M/H^*=AH^*/H^*\cong A/B$. Assume that $R=\sum_{i=1}^my_iK$, where $\{y_1,\ldots ,y_m\}$ is a transversal of $B$ in $A$. Since $A$ normalizes $K^*$ and $B\subseteq K$,  $R$ is a division ring and $[R:K]_r\leq m$. Clearly, $R$ is a division ring generated by $A$ and $K$. Therefore, $M$ normalizes $R$. By the maximality of $M$ in $D^*$,  either $D^*$ normalizes $R$ or $R^*\leq M$. If the second case occurs, then $A$ is a locally nilpotent normal subgroup of $R^*$. By Lemma~ \ref{lem:3.1a}, $A$ is abelian, and  this contradicts to the fact that $K^*$ is the Fitting subgroup of $M$ ~ (see Theorem \ref{thm:3.4} $(ii)$). Thus, $D^*$ normalizes $R$, and by Stuth's theorem,  $R=D$. Therefore,  
$$[D:K]_r\leq m=|M/H^*|=|\mathrm{Gal}(K/F)|=[K:F].$$ 
This  implies  $[D:F]=m^2$, and $K=H$ is a maximal subfield of $D$. From the condition $M/H^*=AH^*/H^*$ and $K^*<A$, it follows that $M=A$. Now, since $F[M]=D$, $M$ is a locally nilpotent absolutely irreducible subgroup of $D^*$, and by \cite[5.7.11 p.~ 215]{wehrfritz}, $M/Z(M)$ is locally finite. Thus, $K/F$ is a nontrivial radical Galois extension. By \cite[15.13, p. 258]{lam}, $F$ is algebraic over a finite field, so is $D$. In view of  Jacobson's theorem \cite[ p. 219]{lam}, $D$ is a field, a contradiction. Thus, $K^*=A$ is the Hirsch-Plotkin radical of $M$.
\end{proof}

The proof of the following corollary is similar to the proof of Corollary \ref{cor:3.5} and it should be omitted. 

\begin{corollary}\label{cor:4.2}
Let $D$ be a division ring with center $F$ and suppose that $M$ is  a maximal subgroup of $D^*$. Assume that $M$ contains an abelian normal subgroup $A$ and there exists some element $\alpha\in A\setminus Z(M)$ such that $\alpha$ is algebraic over $F(Z(M))$. Then, $K=F[A]$ satisfies the conditions (i) - (v) in Theorem \ref{thm:4.1}.
\end{corollary}

\begin{lemma}\label{lem:4.3}
Let $D$ be a division ring with center $F$. Assume that $M$ is a maximal subgroup of $D^*$ such that $C_D(M)=F$ and $F[M]^*=M$. If $M'$ is algebraic over $F$, then $F[M']$ is an algebraic division $F$-algebra and $F[M']^*\trianglelefteq M$.
\end{lemma}
\begin{proof} Firstly, we claim that if $x\in M$ and $g(t)=(t-r_n)\cdots (t-r_1)\in D[t]$ with $r_n,\ldots ,r_1\in x^M$, then $g(x)\in M\cup \{0\}$. Consider $h(t)=(t-r_{n-1})\cdots (t-r_1)$. By induction, we have $h(x)\in M\cup \{0\}$. If $h(x)=0$, then $g(x)=0$ as claimed. If $h(x)\neq~ 0$, then, by \cite[(16.3), p. 263]{lam}, we have $g(x)=(x^{h(x)}-r_n)h(x)$. Take $r_n=x^m$ with $m\in M$, we get  
$$x^{h(x)}-r_n=(x^{m^{-1}h(x)}-x)^m=([m^{-1}h(x),x]-1)^mx^m.$$
On the other hand,  since $M'$ is algebraic over $F$ and $F[M]^*=M$,  it follows that $[m^{-1}h(x),x]-1\in M$. Therefore, $x^{h(x)}-r_n\in M$, so $g(x)\in M$.
\\\indent
Next, by the same argument as in the proof of Theorem \ref{thm:2.2} and by what we have proved, if $x\in M$ is algebraic over $F$ with minimal polynomial $f(t)$ of degree $n$,  then there exist $x_1,\ldots,x_{n-1}\in x^M$ such that $$f(t)=(t-x_{n-1})\cdots (t-x_1)(t-x)\in D[t].$$ Also, by the same argument as in the proof of Corollary \ref{cor:2.4},  we have $x^n\in M'F^*$. Thus,  if $x,y\in M$ are algebraic over $F$, then $xM'F^*$ and $yM'F^*$ are two elements of $M/M'F^*$ of finite order. Therefore, $xyM'F^*$ is of finite order too. By the  supposition, $M'$ is algebraic over $F$. So, $xy$ is also algebraic over $F$. Assume that $x,y\in M$ are algebraic over $F$. Since $F[M]^*=M$, $x+y=x(1+x^{-1}y)\in M$ is algebraic over $F$. Therefore,  $F[M']$ is an algebraic division $F$-algebra, and $F[M']^*\subseteq F[M]^*=M$. This completes our proof.
\end{proof}

In the next theorem, we get some result as in \cite[Theorem 6]{maxn}, but with a weaker condition. In fact, we replace the condition of algebraicity of $M$ by the condition of algebraicity of derived subgroup $M'$.

\begin{theorem}\label{thm:4.4}
Let $D$ be a division ring with center $F$. Assume that $M$ is a non-abelian locally soluble maximal subgroup of $D^*$ such that $M'$ is algebraic over $F$. Then, $[D:F]=p^2$, $M/M'F^*\cong\bigoplus_{i\in I}\mathbb{Z}_p$, where $p$ is a prime number, and there exists a maximal subfield $K$ of $D$ such that $K/F$ is a Galois extension, $\mathrm{Gal}(K/F)\cong M/K^*\cong \mathbb{Z}_p$, $K^*$ is the $FC$-center and the Hirsch-Plotkin radical of $M$. Furthermore, for any $x\in M\setminus K^*$, we have $x^p\in F,\ M=\bigcup_{i=1}^pK^*x^i$ and $D=F[M]=\bigoplus_{i=1}^pKx^i$.
\end{theorem}
\begin{proof} If $F(M)\neq D$, then by the maximality of $M$, $M\cup \{0\}=F(M)$ is a division ring. By Remark 2 in \cite{nilpotent}, $M$ is abelian, a contradiction. Hence, $F(M)=D$, and $C_D(M)=F$. Suppose that $F[M]\neq D$. Then, $F[M]^*=M$ by the maximality of $M$. By Lemma \ref{lem:4.3}, $F[M']$ is a division ring whose  multiplicative group is locally soluble.  In view of  Remark 2 in \cite{nilpotent},  $M'$ is abelian. So, by Theorem \ref{thm:3.6}, we have $F[M]=D$, a contradiction. Therefore,  $M$ is absolutely irreducible, and by  \cite[Corollary 4]{ls},  $M$ is abelian-by-locally finite. Thus,  there exists an abelian normal subgroup $A$ of $M$ such that $M/A$ is locally finite. We need to show that $D$ is locally finite over $F$. Now, let us examine two possible cases.
\\[5pt]\textit{Case 1}: $A\cap M'\subseteq F$
\\\indent
We note that $F(M')^*$ is normalized by $M$. By the maximality of $M$,  either $F(M')\cap G\subseteq M$ or $G$ normalizes $F(M')^*$. In the first case, $M'$ is a subnormal subgroup of $F(M')^*$. On the other hand, $M$ is abelian-by-locally finite, so is $M'$. By Lemma~\ref{lem:3.1}, $M'$ is abelian.  Hence, by Theorem~\ref{thm:3.6}, $[D:F]<\infty$, as claimed. In the second case, $G$ normalizes $F(M')^*$, and by Stuth's theorem,  either $F(M')\subseteq F$ or $F(M')=D$. If $F(M')\subseteq F$, then, by Theorem \ref{thm:3.6}, $[D:F]<\infty$, and we are done. Suppose that $F(M')=D$. We know that $M'/M'\cap A\cong AM'/A\leq M/A$ is locally finite, so $M'/M'\cap F$ is locally finite since $M'\cap A\subseteq F$. Therefore,  $D=F(M')$ is locally finite over $F$.
\\[5pt]\textit{Case 2}: $A\cap M'\not\subseteq F$
\\\indent
Since $M'$ is algebraic over $F$, there exists an element $x\in A\cap M'\setminus F$ which is algebraic over $F$. By Corollary \ref{cor:4.2},  there exists a subfield $K_1$ of $D$ such that  $K_1$ satisfies the conditions $(i)-(v)$ in Theorem \ref{thm:4.1}. So, $C_D(K_1)^*\leq M$ is abelian-by-locally finite, and by Lemma \ref{lem:3.1},  $C_D(K_1)^*$ is abelian. Now,  by Double Centralizer theorem,  $[D:F]<\infty$. Therefore, $D$ is locally finite. Since $M$ is locally soluble, $M$ contains no non-cyclic free subgroups. By Theorem \ref{thm:3.15},  there exists a maximal subfield $K$ of $D$ as in Theorem \ref{thm:3.15}. Then,  $M/K^*$ is a finite simple group. Recall that, by the supposition, $M$ is locally soluble. Hence, $M/K^*\cong\mathbb{Z}_p$, where $p$ is a prime number. Thus,  $M'\leq K^*$ and $M$ is metabelian. Now, the proof is complete by applying of theorems~\ref{thm:3.6} and \ref{thm:4.1}. 
\end{proof}

Using the results obtained above in this section, we can now improve one previous result by B. X. Hai and N. V. Thin. In fact, in \cite[Theorem 3.2]{hai-thin1},  it was proved that if $D$ is a non-commutative division ring which is algebraic over its center, then any locally nilpotent maximal subgroup of $D^*$ is the multiplicative group of some maximal subfield of $D$. Here, we are now ready to improve this result by the following theorem.

\begin{theorem}\label{thm:4.5}
Let $D$ be a non-commutative division ring with center $F$ and suppose $M$ is  a locally nilpotent maximal subgroup of $D^*$. If $M'$ is algebraic over $F$, then $M$ is abelian. Consequently, $M$ is the multiplicative group of some maximal subfield of $D$.
\end{theorem}
\begin{proof} Assume that $M$ is non-abelian. Using Theorem \ref{thm:4.4}, we conclude that the Hirsch-Plotkin radical of $M$ is a proper subgroup of $M$, this contradicts to the fact that $M$ is locally nilpotent. Thus, $M$ is abelian. The last conclusion is evident.
\end{proof}

\noindent
{\bf Acknowledgments}

The authors thank the Editor and the referee for their  useful remarks and comments. 
The first author is funded by Vietnam National Foundation for Science and Technology Development (NAFOSTED) under Grant No. 101.04-2013.01. A part of the first revised version of this work was done when he was working as a researcher at the Vietnam Institute for Advanced Study in Mathematics (VIASM).  He would like to express his sincere thanks to VIASM for providing a fruitful research environment and hospitality. 


\end{document}